\documentclass{article}

\usepackage{amsmath,amssymb,amsthm}

\theoremstyle{plain}
\newtheorem{lem}{Lemma}[section]
\newtheorem{thm}[lem]{Theorem}

\makeatletter
  \@addtoreset{equation}{section}
\makeatother

\allowdisplaybreaks

\title{The $2$-components of the $31$-stem homotopy groups of the $9$
and $10$-spheres}
\author{Tomohisa Inoue and Juno Mukai}
\date{}

\begin{document}

\maketitle

\begin{abstract}
  Group structures of the $2$-primary components of the $31$-stem
  homotopy groups of spheres were studied by Oda in 1979. There are,
  however, two incompletely determined groups. In this paper, our
  investigation with Toda's composition method gives structures of
  them.
\end{abstract}

\section{Introduction.}

We denote by $\pi_k^n$ the direct sum of the torsion-free part and the
$2$-primary component of the $k$-th homotopy group $\pi_k(S^n)$ of the
$n$-dimensional sphere $S^n$. The group $\pi_k^n$ was studied by Toda
\cite{Tod62} with his composition method, and several authors
\cite{MT63, Mim65, MMO75, Oda79} followed the method. In particular,
Oda \cite{Oda79} investigated the $2$-primary components of $k$-stem
homotopy groups for $25\leq k\leq31$. There are, however, two
incompletely determined groups in 31-stem: $\pi_{40}^9$ and
$\pi_{41}^{10}$.

We denote by $\{\chi_1,\dotsc,\chi_n\}$ a group generated by elements
$\chi_1,\dotsc,\chi_n$. If the group is isomorphic to a group $G$, it
is denoted by $G\{\chi_1,\dotsc,\chi_n\}$. For the group $\mathbb Z$
of integers, we set $\mathbb Z_n=\mathbb Z/n\mathbb Z$, and let
$(\mathbb Z_n)^k$ be the direct sum
$\mathbb Z_n\oplus\dotsb\oplus\mathbb Z_n$ of $k$-copies of
$\mathbb Z_n$. Let $E\colon\pi_k^n\rightarrow\pi_{k+1}^{n+1}$ be the
suspension homomorphism, and let
$P\colon\pi_{k+2}^{2n+1}\rightarrow\pi_k^n$ be the $P$-homomorphism
which is denoted by $\Delta$ in \cite{Tod62}. We use generators of
homotopy groups of spheres in the same symbol as defined in
\cite{Tod62, MT63, Mim65, MMO75, Oda79}.

Oda \cite[Theorem 3(c)]{Oda79} showed
\begin{equation}
  \label{Oda_p4009}
  \begin{split}
    \pi_{40}^9={}
    & \mathbb Z_2\{\sigma_9\delta_{16}\}\oplus
      \mathbb Z_2\{\sigma_9\bar\mu_{16}\sigma_{33}\}\oplus
      \mathbb Z_2\{\sigma_9\bar\sigma_{16}'\}\oplus
      \mathbb Z_2\{\delta_9\sigma_{33}\}\oplus
      {}\\
    & \mathbb Z_{16}\{\alpha_3^{IV}\}\oplus
      G_1\{\bar\nu_9\nu_{17}\bar\kappa_{20}\}
  \end{split}
\end{equation}
and
\begin{equation}
  \label{Oda_p4110}
  \pi_{41}^{10}=
  \mathbb Z_8\{P\sigma_{21}^\ast\}\oplus
  \mathbb Z_{16}\{E\alpha_3^{IV}\}\oplus
  G_2\{\kappa_{10}^\ast,\delta_{10}\sigma_{34}\},
\end{equation}
where $G_1$ is $\mathbb Z_2$ or $0$, and $G_2$ is $(\mathbb Z_2)^2$ or
$\mathbb Z_4$. We mention that equations
\begin{equation}
  \label{imageP}
  P(\nu_{19}\bar\kappa_{22})=\bar\nu_9\nu_{17}\bar\kappa_{20}
  \quad\text{and}\quad
  P(\nu_{21}\bar\sigma_{24})=\delta_{10}\sigma_{34}
\end{equation}
are obtained in \cite[III-(10.3), Proposition 4.7(6)]{Oda79}. We
denote by $[\chi_1,\chi_2]$ the Whitehead product of $\chi_1$ and
$\chi_2$, and by $\iota_n\in\pi_n^n$ the homotopy class of the
identity map on $S^n$. Since $PE^2\chi=\pm[\iota_n,\iota_n]\circ\chi$
for an element $\chi\in\pi_k^{2n-1}$ (see
\cite[Proposition 2.5]{Tod62}), we have
$PE^{n+2}\chi=\pm[\iota_n,\iota_n]E^n\chi=\pm[\iota_n,E\chi]$ for an
element $\chi\in\pi_k^{n-1}$. Then images of elements by $P$ are
frequently written as Whitehead products. For example, the equations
of \eqref{imageP} are identical with
\begin{equation}
  \label{imageW}
  [\iota_9,\nu_9\bar\kappa_{12}]=\bar\nu_9\nu_{17}\bar\kappa_{20}
  \quad\text{and}\quad
  [\iota_{10},\nu_{10}\bar\sigma_{13}]=\delta_{10}\sigma_{34}.
\end{equation}
Signs are unnecessary because $2\bar\nu_9\nu_{17}\bar\kappa_{20}=0$
and $2\delta_{10}\sigma_{34}=0$.

The purpose of this note is to determine the groups \eqref{Oda_p4009}
and \eqref{Oda_p4110}. Our results are stated as follows.

\begin{thm}
  \label{IM}
  \begin{enumerate}
    \item\label{IM_p4009}
      $[\iota_9,\nu_9\bar\kappa_{12}]
      =\bar\nu_9\nu_{17}\bar\kappa_{20}=0$ and
      \begin{equation*}
        \pi_{40}^9=
        \mathbb Z_2\{\sigma_9\delta_{16}\}\oplus
        \mathbb Z_2\{\sigma_9\bar\mu_{16}\sigma_{33}\}\oplus
        \mathbb Z_2\{\sigma_9\bar\sigma_{16}'\}\oplus
        \mathbb Z_2\{\delta_9\sigma_{33}\}\oplus
        \mathbb Z_{16}\{\alpha_3^{IV}\}.
      \end{equation*}
    \item\label{IM_p4110}
      $[\iota_{10},\nu_{10}\bar\sigma_{13}]
      =\delta_{10}\sigma_{34}=2\kappa_{10}^\ast$ and
      \begin{equation*}
        \pi_{41}^{10}=
        \mathbb Z_8\{P\sigma_{21}^\ast\}\oplus
        \mathbb Z_{16}\{E\alpha_3^{IV}\}\oplus
        \mathbb Z_4\{\kappa_{10}^\ast\}.
      \end{equation*}
  \end{enumerate}
\end{thm}

To show Theorem \ref{IM}, we use Toda's composition method which
requires generators of homotopy groups. In Section 2, we describe some
relations of generators. Section 3 gives details of our investigation
concerning $\pi_{40}^9$. The structure of $\pi_{41}^{10}$ is obtained
in Section 4.

At the end of this section, we recall \cite[Proposition 4.2]{Tod62}:
there exists an exact sequence
\begin{equation*}
  \dotsb\rightarrow
  \pi_{k+2}^{2n+1}\xrightarrow P
  \pi_k^n\xrightarrow E
  \pi_{k+1}^{n+1}\xrightarrow H
  \pi_{k+1}^{2n+1}\xrightarrow P
  \pi_{k-1}^n\xrightarrow E
  \pi_k^{n+1}\rightarrow
  \dotsb,
\end{equation*}
called the EHP sequence, where $H$ be the Hopf homomorphism.

\section{Recollection of some relations.}

We use notations in \cite{Tod62} and properties of Toda brackets
freely. We know
\begin{gather}
  \label{w2}
  2\eta_2=[\iota_2,\iota_2],
  \quad\text{\cite[Proposition 5.1]{Tod62};}
  \\
  \label{w4}
  \eta_3\nu_4=\nu'\eta_6,
  \quad\eta_4\nu_5=(E\nu')\eta_7=[\iota_4,\eta_4],
  \quad\text{\cite[(5.9), (5.11)]{Tod62};}
  \\
  \label{w5}
  \nu_5\eta_8=[\iota_5,\iota_5],
  \quad\text{\cite[(5.10)]{Tod62};}
  \\
  \label{w8}
  2\sigma_8-E\sigma'=\pm[\iota_8,\iota_8],
  \quad\text{\cite[(5.16)]{Tod62};}
  \\
  \label{w9}
  \eta_9\sigma_{10}+\sigma_9\eta_{16}
  =[\iota_9,\iota_9],
  \quad\text{\cite[Lemma 6.4, (7.1)]{Tod62};}
  \\
  \label{et7si8}
  \eta_7\sigma_8=\bar\nu_7+\varepsilon_7+\sigma'\eta_{14},
  \quad\eta_9\sigma_{10}=\bar\nu_9+\varepsilon_9,
  \quad\text{\cite[Lemma 6.4, (7.4)]{Tod62};}
  \\
  \label{et6ep7}
  \eta_3\varepsilon_4=\varepsilon_3\eta_{11},
  \quad\text{\cite[(7.5)]{Tod62};}
  \\
  \label{ep4nu12}
  \varepsilon_4\nu_{12}=[\iota_4,\iota_4]\bar\nu_7,
  \quad\text{\cite[(7.13)]{Tod62};}
  \\
  \label{nu6ep9}
  \nu_6\bar\nu_9=\nu_6\varepsilon_9
  =2\bar\nu_6\nu_{14}=[\iota_6,\nu_6^2],
  \quad\text{\cite[(7.17), (7.18)]{Tod62};}
  \\
  \label{w15}
  2\sigma_{15}^2=[\iota_{15},\iota_{15}],
  \quad\text{\cite[(10.10)]{Tod62};}
  \\
  \label{ep3si11}
  \varepsilon_3\sigma_{11}=0,
  \quad\sigma_{11}\varepsilon_{18}=0,
  \quad\bar\nu_6\sigma_{14}=0,
  \quad\text{\cite[Lemma 10.7]{Tod62};}
  \\
  \label{et6ka7}
  \eta_6\kappa_7=\bar\varepsilon_6,
  \quad\kappa_9\eta_{23}=\bar\varepsilon_9,
  \quad\text{\cite[(10.23)]{Tod62};}
  \\
  \label{ep3ep11}
  \nu_5\sigma_8\nu_{15}^2=\eta_5\bar\varepsilon_6,
  \quad
  \varepsilon_3^2=\varepsilon_3\bar\nu_{11}
  =\eta_3\bar\varepsilon_4=\bar\varepsilon_3\eta_{18},
  \quad\text{\cite[Lemma 12.10]{Tod62}.}
\end{gather}
By \eqref{et7si8} and \eqref{ep3si11}, we have
\begin{equation}
  \label{et9si10si17}
  \eta_9\sigma_{10}^2=(\bar\nu_9+\varepsilon_9)\sigma_{17}
  =\bar\nu_9\sigma_{17}+\varepsilon_9\sigma_{17}=0.
\end{equation}

We recall \cite[(7.19)]{Tod62}:
\begin{equation}
  \label{nu7si10}
  \sigma'\nu_{14}=x\nu_7\sigma_{10}
  \text{ for some odd integer }x.
\end{equation}
Then \eqref{w8} gives $2\sigma_9\nu_{16}=x\nu_9\sigma_{12}$. Since
$8\sigma_9\nu_{16}=8\nu_9\sigma_{12}=0$ (see
\cite[Theorem 7.3]{Tod62}), we have
\begin{equation}
  \label{nu9si12}
  \nu_9\sigma_{12}=\pm2\sigma_9\nu_{16}
\end{equation}
and hence, by \cite[(7.21)]{Tod62},
\begin{equation}
  \label{w11}
  \nu_{10}\sigma_{13}=2\sigma_{10}\nu_{17}=[\iota_{10},\eta_{10}],
  \quad
  \sigma_{11}\nu_{18}=[\iota_{11},\iota_{11}].
\end{equation}
Similarly, we obtain
\begin{equation*}
  \label{ze9si20}
  \zeta_9\sigma_{20}=\pm2\sigma_9\zeta_{16}
\end{equation*}
because $\sigma'\zeta_{14}=x\zeta_7\sigma_{18}$ for some odd integer
$x$ (see \cite[Lemma 12.12]{Tod62}) and
$8\sigma_9\zeta_{16}=8\zeta_9\sigma_{20}=0$ (see
\cite[Theorem 12.8]{Tod62}). The relation and \cite[(12.25)]{Tod62}
imply
\begin{equation}
  \label{ze10si21}
  \zeta_{10}\sigma_{21}=2\sigma_{10}\zeta_{17}
  =[\iota_{10},\mu_{10}].
\end{equation}

By \cite[Propositions (2.2)(5), (2.2)(6)]{Ogu64} and
\cite[(7.25)]{Tod62}, we have
\begin{equation}
  \label{et5ze6}
  \eta_5\zeta_6=0,
  \quad
  \zeta_6\eta_{17}=\nu_6\mu_9=8[\iota_6,\iota_6]\sigma_{11}.
\end{equation}

By \cite[Lemma 5.12 and Remark]{Tod62},
\begin{equation}
  \label{nu6nu9}
  \{\eta_n,\nu_{n+1},\eta_{n+4}\}=\nu_n^2
  \text{ for }n\geq5,
\end{equation}
and hence, by
$\bar\sigma_6\in\{\nu_6,\varepsilon_9+\bar\nu_9,\sigma_{17}\}_1$
(see \cite[p.~138]{Tod62}), \eqref{et7si8}, \eqref{nu7si10} and
\eqref{w11},
\begin{equation*}
  \begin{split}
    \eta_5\bar\sigma_6
    & \in\eta_5\circ\{\nu_6,\varepsilon_9+\bar\nu_9,\sigma_{17}\}\\
    & =\eta_5\circ\{\nu_6,\eta_9\sigma_{10},\sigma_{17}\}\\
    & =-\{\eta_5,\nu_6,\eta_9\sigma_{10}\}\circ\sigma_{18}\\
    & \supset-\{\eta_5,\nu_6,\eta_9\}\circ\sigma_{11}^2
      =-\nu_5^2\sigma_{11}^2
      =-x\nu_5(E\sigma')\nu_{15}\sigma_{18}
      =0\\
    & \bmod\eta_5\nu_6\circ\pi_{25}^9
      +\eta_5\circ\pi_{18}^6\circ\sigma_{18}
  \end{split}
\end{equation*}
for some odd integer $x$. Here, $\eta_5\nu_6=0$ by \eqref{w4};
$\eta_5\circ\pi_{18}^6=\{\eta_5[\iota_6,\iota_6]\sigma_{11}\}$ by
\cite[Theorem 7.6]{Tod62}. By \eqref{w5} and by relations
$\eta_5^3=4\nu_5$ and $2\nu_5^2=0$ (see
\cite[(5.5), Proposition 5.11]{Tod62}), we have
$\eta_5[\iota_6,\iota_6]=
[\eta_5,\eta_5]=
[\iota_5,\iota_5]\eta_9^2=
\nu_5\eta_8^3=4\nu_5^2=0$.
This implies $\eta_5\circ\pi_{18}^6=0$. Thus
\begin{equation}
  \label{et5bs6}
  \eta_5\bar\sigma_6=0.
\end{equation}

We know \cite[Lemma 16.1, p.~50]{MT63}:
\begin{equation}
  \label{p3010}
  \pi_{30}^{10}=
  \mathbb Z_8\{\bar\kappa_{10}\}\oplus\mathbb Z_8\{\beta'\},
\end{equation}
where
\begin{equation}
  \label{b'}
  E\beta'=\theta'\varepsilon_{23},
  \quad
  H\beta'=\zeta_{19}
  \quad\text{and}\quad
  2\beta'=\pm[\iota_{10},\zeta_{10}].
\end{equation}
By \cite[Proposition (2.6)(2)]{Ogu64},
$\nu_8\theta'=(E\sigma')\varepsilon_{15}$ or $(E\sigma')\bar\nu_{15}$.
Then $E(\nu_7\beta')=\nu_8\theta'\varepsilon_{23}$ equals
$(E\sigma')\varepsilon_{15}^2$ or
$(E\sigma')\bar\nu_{15}\varepsilon_{23}$. Here,
\begin{equation}
  \label{ep9ep17}
  \varepsilon_9^2=\eta_9\bar\varepsilon_{10}
  =\nu_9\sigma_{12}\nu_{19}^2=0
\end{equation}
by \eqref{ep3ep11} and \eqref{w11}. Furthermore,
$\bar\nu_6\varepsilon_{14}=0$ by \cite[Proposition (2.8)(2)]{Ogu64}.
These imply $E(\nu_7\beta')=0$ and
\begin{equation}
  \label{nu7b'}
  \nu_7\beta'=0
\end{equation}
because $E\colon\pi_{30}^7\rightarrow\pi_{31}^8$ is a monomorphism. On
the other hand,
$E(\beta'\sigma_{30})=\theta'\varepsilon_{23}\sigma_{31}=0$ by
\eqref{ep3si11} and $H(\beta'\sigma_{30})=\zeta_{19}\sigma_{30}=0$ by
\eqref{ze10si21}. Then
$\beta'\sigma_{30}\in\ker E=\mathbb Z_4\{P\nu_{21}^\ast\}$ (see
\cite[Theorem 2(a)]{Oda79}) and $\beta'\sigma_{30}\in\ker H$. Since
$HP\nu_{21}^\ast=\pm2\nu_{19}^\ast$ (see
\cite[Proposition 2.7]{Tod62}) and $\nu_{19}^\ast$ is of order $8$
(see \cite[Theorem 12.22]{Tod62}), we have $\ker E\cap\ker H=0$, that
is,
\begin{equation}
  \label{b'si30}
  \beta'\sigma_{30}=0.
\end{equation}

By \cite[the proof of Lemma 3.2]{MMpre}, we have
$(E\lambda')\eta_{30}\sigma_{31}=(E\xi')\eta_{30}\sigma_{31}=0$. Since
\begin{equation}
  \label{mono}
  E\colon\pi_{37}^{11}\rightarrow\pi_{38}^{12}
  \text{ is a monomorphism,}
  \quad\text{\cite[I-(8.17)]{Oda79},}
\end{equation}
the relations give
\begin{equation}
  \label{pret29si30}
  \lambda'\eta_{29}\sigma_{30}=\xi'\eta_{29}\sigma_{30}=0.
\end{equation}
We may show \eqref{pret29si30} by the same way as \cite{MMpre} without
\eqref{mono}.

We show the following lemma.

\begin{lem}
  \label{bm3ka20}
  $\bar\mu_3\kappa_{20}=0$.
\end{lem}
\begin{proof}
  Since $\bar\mu_3\in\{\mu_3,2\iota_{12},8\sigma_{12}\}_1$ (see
  \cite[p.~136]{Tod62}), we have
  \begin{equation*}
    \bar\mu_3\kappa_{20}
    \in\{\mu_3,2\iota_{12},8\sigma_{12}\}_1\circ\kappa_{20}
    =\mu_3\circ E\{2\iota_{11},8\sigma_{11},\kappa_{18}\}
    \subset\mu_3\circ E\pi_{33}^{11}.
  \end{equation*}
  Then
  $\pi_{33}^{11}=\{\sigma_{11}\rho_{18},
  \varepsilon_{11}\kappa_{19},\nu_{11}\bar\sigma_{14}\}$
  (see \cite[Theorem B]{Mim65}) gives
  \begin{equation*}
    \bar\mu_3\kappa_{20}
    \in\{\mu_3\sigma_{12}\rho_{19},
    \mu_3\varepsilon_{12}\kappa_{20},\mu_3\nu_{12}\bar\sigma_{15}\}.
  \end{equation*}
  By \cite[III-Proposition 2.6(1)]{Oda79}, the element
  $\mu_3\sigma_{12}\rho_{19}$ is equal to $0$. Since
  $\eta_3\mu_4=\mu_3\eta_{12}$,
  $\mu_3\varepsilon_{12}\equiv\eta_3\mu_4\sigma_{13}
  \bmod2\bar\varepsilon'$
  (see \cite[Propositions (2.2)(2), (2.13)(7)]{Ogu64}), and
  $\kappa_{10}$ is of order $2$ (see \cite[Theorem 10.3]{Tod62}), we
  have
  \begin{equation*}
    \mu_3\varepsilon_{12}\kappa_{20}\equiv
    \eta_3\mu_4\sigma_{13}\kappa_{20}=
    \mu_3\eta_{12}\sigma_{13}\kappa_{20}\bmod
    (2\bar\varepsilon')\kappa_{20}=
    \bar\varepsilon'(2\kappa_{20})=0.
  \end{equation*}
  Then
  $\mu_3\varepsilon_{12}\kappa_{20}=
  \mu_3\eta_{12}\nu_{13}E^3\lambda=0$,
  as $\sigma_{10}\kappa_{17}=\nu_{10}\lambda$ (see
  \cite[I-Proposition 3.1(1)]{Oda79}) and by \eqref{w4}. By relations
  $\mu_3\nu_{12}=\nu'\eta_6\varepsilon_7$ (see
  \cite[Proposition (2.2)(4)]{Ogu64}), \eqref{et6ep7} and
  \eqref{et5bs6}, we have
  $\mu_3\nu_{12}\bar\sigma_{15}=
  \nu'\eta_6\varepsilon_7\bar\sigma_{15}=
  \nu'\varepsilon_6\eta_{14}\bar\sigma_{15}=0$. Thus
  $\bar\mu_3\kappa_{20}=0$.
\end{proof}

\section{Proof of Theorem \ref{IM}\ref{IM_p4009}.}

In this section, we shall show the following theorem.

\begin{thm}
  \label{main1}
  $\sigma'\delta_{14}
  \equiv\bar\nu_7\nu_{15}\bar\kappa_{18}
  \bmod\sigma'\bar\sigma_{14}',
  \sigma'\bar\mu_{14}\sigma_{31}$.
\end{thm}

By \eqref{w8}, $E^2(\sigma'\delta_{14})$,
$E^2(\sigma'\bar\sigma_{14}')$ and
$E^2(\sigma'\bar\mu_{14}\sigma_{31})$ are equal to
$2\sigma_9\delta_{16}$, $2\sigma_9\bar\sigma_{16}'$ and
$2\sigma_9\bar\mu_{16}\sigma_{33}$, respectively. By
\eqref{Oda_p4009}, these elements are $0$. Then Theorem \ref{main1}
gives $\bar\nu_9\nu_{17}\bar\kappa_{20}=0$: it is suffice to show
Theorem \ref{main1} to obtain Theorem \ref{IM}\ref{IM_p4009}.

\begin{lem}
  \label{l_nezs}
  $\langle\nu,\eta,\zeta\rangle=0$.
\end{lem}
\begin{proof}
  In this proof, we refer to the Bott periodicity theorem
  \cite{Bot59}: if $k+1<n$, $\pi_k(SO(n))$ is isomorphic to
  $\mathbb Z_2$, $\mathbb Z_2$, $0$, $\mathbb Z$, $0$, $0$, $0$ and
  $\mathbb Z$ for $k\equiv0$, $1$, $2$, $3$, $4$, $5$, $6$ and
  $7\bmod8$, respectively. For a sufficiently large $n$, we have
  $\zeta_n=J\zeta_n'$, where $\zeta_n'$ is the generator of
  $\pi_{11}(SO(n))\cong\mathbb Z$ and
  $J\colon\pi_{11}(SO(n))\rightarrow\pi_{n+11}(S^n)$ is the
  $J$-homomorphism. A Toda bracket
  $\{\zeta_n',\eta_{11},\nu_{12}\}\subset
  \pi_{16}(SO(n))\cong\mathbb Z_2$
  is well defined because $\zeta_n'\eta_{11}\in\pi_{12}(SO(n))=0$ and
  $\eta_5\nu_6=0$, by \eqref{w4}. By \eqref{nu6nu9} and
  $\pi_{14}(SO(n))=0$, we have
  \begin{equation*}
    \{\zeta_n',\eta_{11},\nu_{12}\}\circ\eta_{16}
    =-(\zeta_n'\circ\{\eta_{11},\nu_{12},\eta_{15}\})
    =-\zeta_n'\nu_{11}^2
    \in\pi_{14}(SO(n))\circ\nu_{14}
    =0.
  \end{equation*}
  By \cite[Lemma 2]{Ker60},
  $\eta_{16}^\ast\colon\pi_{16}(SO(n))\rightarrow\pi_{17}(SO(n))$ is
  an isomorphism, and hence $\{\zeta_n',\eta_{11},\nu_{12}\}=0$. Then,
  by \cite[Lemma 5.1]{MMpre}, we have
  \begin{equation*}
    \{\zeta_n,\eta_{n+11},\nu_{n+12}\}
    \supset(-1)^nJ\{\zeta_n',\eta_{11},\nu_{12}\}=0.
  \end{equation*}
  Therefore, groups $\pi_5^s=0$ and $\pi_{13}^s=0$ (see
  \cite[Proposition 5.9, Theorem 7.7]{Tod62}) shows that
  $\langle\zeta,\eta,\nu\rangle=\langle\nu,\eta,\zeta\rangle\ni0
  \bmod\nu\circ\pi_{13}^s+\pi_5^s\circ\zeta=0$.
\end{proof}

Toda brackets $\{\nu_6,\eta_9,\zeta_{10}\}_1$,
$\{\eta_{10},\zeta_{11},\sigma_{22}\}$ and
$\{\zeta_{11},\sigma_{22},2\sigma_{29}\}$ are well defined:
$\nu_6\eta_9=0$ by \eqref{w5}; $\eta_5\zeta_6=0$ by \eqref{et5ze6};
$\zeta_{11}\sigma_{22}=0$ by \eqref{ze10si21}; and $2\sigma_{16}^2=0$
by \eqref{w15}.

\begin{lem}
  \label{l_nez}
  $\{\nu_6,\eta_9,\zeta_{10}\}\ni\zeta'
  \bmod\eta_6\bar\varepsilon_7,2\zeta'$
  and
  $\{\nu_7,\eta_{10},\zeta_{11}\}
  \ni\sigma'\eta_{14}\varepsilon_{15}
  \bmod\eta_7\bar\varepsilon_8$.
\end{lem}
\begin{proof}
  By \cite[Proposition 5.9, Theorem 7.7]{Tod62}, we have
  $\pi_{11}^6=\{P\iota_{13}\}$, $\pi_{12}^7=0$,
  $\pi_{n+13}^n=\{\sigma_n\nu_{n+7}^2\}$ for $n=9,10$ and
  $E\pi_{21}^8=\pi_{22}^9$. The last equation implies that the Toda
  bracket $\{\nu_6,\eta_9,\zeta_{10}\}$ is equal to
  $\{\nu_6,\eta_9,\zeta_{10}\}_1$. The indeterminacy of
  $\{\nu_6,\eta_9,\zeta_{10}\}$ is
  $\nu_6\circ\pi_{22}^9+\pi_{11}^6\circ\zeta_{11}
  =\{\nu_6\sigma_9\nu_{16}^2,P\zeta_{13}\}$.
  Here, $\nu_6\sigma_9\nu_{16}^2=\eta_6\bar\varepsilon_7$ by
  \eqref{ep3ep11} and $P\zeta_{13}=\pm2\zeta'$ by
  \cite[(12.4)]{Tod62}. The indeterminacy of
  $\{\nu_7,\eta_{10},\zeta_{11}\}$ is obtained in the same way. By
  \eqref{w5}, we have $P\iota_{11}=\nu_5\eta_8$ and hence the kernel
  of $P\colon\pi_{11}^{11}\rightarrow\pi_9^5
  =\mathbb Z_2\{\nu_5\eta_8\}$
  (see \cite[Proposition 5.8]{Tod62}) is generated by $2\iota_{11}$.
  Then an equation
  \begin{equation*}
    H\{\nu_6,\eta_9,\zeta_{10}\}_1
    =-P^{-1}(\nu_5\eta_8)\circ\zeta_{11}
    =\{a\zeta_{11}\mid a\text{ is odd}\}
  \end{equation*}
  is obtained by the use of \cite[Proposition 2.6]{Tod62}. Since
  $\pi_{22}^{11}=\mathbb Z_8\{\zeta_{11}\}$,
  $\pi_{22}^6=\mathbb Z_8\{\zeta'\}\oplus
  \mathbb Z_2\{\mu_6\sigma_{15}\}\oplus
  \mathbb Z_2\{\eta_6\bar\varepsilon_7\}$,
  and $H\zeta'\equiv\zeta_{11}\bmod2\zeta_{11}$ (see
  \cite[Theorems 7.4, 12.6, Lemma 12.1]{Tod62}), we have
  \begin{equation*}
    H^{-1}\{a\zeta_{11}\mid a\text{ is odd}\}=
    \{b\zeta'+c\mu_6\sigma_{15}+d\eta_6\bar\varepsilon_7\mid
    b=1,3,5,7,\ c=0,1,\ d=0,1\}.
  \end{equation*}
  This yields
  \begin{equation*}
    \{\nu_6,\eta_9,\zeta_{10}\}\ni\zeta'+c\mu_6\sigma_{15}
  \end{equation*}
  for $c=0$ or $1$, because the indeterminacy of
  $\{\nu_6,\eta_9,\zeta_{10}\}$ is
  $\{\eta_6\bar\varepsilon_7,2\zeta'\}$ as above. Therefore, relations
  $E\zeta'=\sigma'\eta_{14}\varepsilon_{15}$ and $2\varepsilon_{15}=0$
  (see \cite[(12.4), Theorem 7.1]{Tod62}) give
  \begin{equation*}
    \{\nu_7,\eta_{10},\zeta_{11}\}
    \supset-E\{\nu_6,\eta_9,\zeta_{10}\}
    \ni-E\zeta'-c\mu_7\sigma_{16}
    =\sigma'\eta_{14}\varepsilon_{15}+c\mu_7\sigma_{16}.
  \end{equation*}
  By \eqref{w8}, we have
  $E^2(\sigma'\eta_{14}\varepsilon_{15})=
  \sigma_9\eta_{16}\circ2\varepsilon_{17}=0$
  and hence $\langle\nu,\eta,\zeta\rangle\ni c\mu\sigma$. Since
  $\mu\sigma=\sigma\mu\neq0$ (see \cite[p.~156,
  Theorem 12.16]{Tod62}), Lemma \ref{l_nezs} leads to $c=0$.
\end{proof}

\begin{lem}
  \label{l_ezs}
  $\nu_7\circ\{\eta_{10},\zeta_{11},\sigma_{22}\}=0$ and
  $\{\eta_{10},\zeta_{11},\sigma_{22}\}\circ\sigma_{30}=0$.
\end{lem}
\begin{proof}
  By Lemma \ref{l_nez}, we have
  \begin{equation*}
    \nu_7\circ\{\eta_{10},\zeta_{11},\sigma_{22}\}
    =-(\{\nu_7,\eta_{10},\zeta_{11}\}\circ\sigma_{23})
    \subset\{\sigma'\eta_{14}\varepsilon_{15},
    \eta_7\bar\varepsilon_8\}\circ\sigma_{23}.
  \end{equation*}
  Then \eqref{ep3si11} and
  \begin{equation}
    \label{be3si18}
    \bar\varepsilon_3\sigma_{18}=0,
    \quad\text{\cite[(2.4)]{MMO75},}
  \end{equation}
  give the first half. By \eqref{p3010},
  $\{\eta_{10},\zeta_{11},\sigma_{22}\}\subset
  \mathbb Z_8\{\bar\kappa_{10}\}\oplus\mathbb Z_8\{\beta'\}$.
  We have $\nu_7\bar\kappa_{10}$ is of order $8$ by
  \cite[Theorem 1.1(a)]{MMO75}, and $\nu_7\beta'=0$ by \eqref{nu7b'}.
  Then the first half shows
  $\{\eta_{10},\zeta_{11},\sigma_{22}\}\subset\mathbb Z_8\{\beta'\}$.
  Hence, by \eqref{b'si30}, we obtain the second half.
\end{proof}

By Lemma \ref{l_ezs}, we may consider the Jacobi identity:
\begin{equation}
  \label{e_jac1}
  \begin{split}
    \{\{\nu_7,\eta_{10},\zeta_{11}\},\sigma_{23},2\sigma_{30}\}
    & +\{\nu_7,\{\eta_{10},\zeta_{11},\sigma_{22}\},2\sigma_{30}\}\\
    & +\{\nu_7,\eta_{10},\{\zeta_{11},\sigma_{22},2\sigma_{29}\}\}
      \ni0.
  \end{split}
\end{equation}
As a preparation, we show the next lemma, which is followed by
observations of Toda brackets in \eqref{e_jac1}.

\begin{lem}
  \label{l_ind}
  $\pi_{31}^7\circ2\sigma_{31}=0$ and $\nu_7\circ\pi_{38}^{10}=0$.
\end{lem}
\begin{proof}
  Since $\pi_{31}^7\cong(\mathbb Z_2)^7$ (see
  \cite[Theorem 1.1(b)]{MMO75}), the first half of the lemma is
  obvious. The group $\pi_{38}^{10}$ is obtained in
  \cite[Theorem 2(b)]{Oda79}:
  \begin{equation*}
    \pi_{38}^{10}=
    \mathbb Z_8\{F_1^{(1)}\}\oplus
    \mathbb Z_2\{\sigma_{10}^4\}\oplus
    \mathbb Z_2\{\sigma_{10}\nu_{17}^\ast\nu_{35}\}\oplus
    \mathbb Z_2\{\bar\nu_{10}\bar\kappa_{18}\}\oplus
    \mathbb Z_2\{\varepsilon_{10}\bar\kappa_{18}\},
  \end{equation*}
  where $F_1^{(1)}\in\{F_1,8\iota_{30},2\sigma_{30}\}_1$ for
  $F_1\in\{P\iota_{21},\eta_{19}^2,\varepsilon_{21}\}_1$ (see
  \cite[Definition 3.18]{Oda76}). Since $F_1=x\beta'$ for some odd
  integer $x$ (see \cite[I-(2.1)]{Oda79}), \eqref{nu7b'} and the first
  half lead to
  \begin{equation*}
    \nu_7F_1^{(1)}
    \in\nu_7\circ\{x\beta',8\iota_{30},2\sigma_{30}\}
    =-\{\nu_7,x\beta',8\iota_{30}\}\circ2\sigma_{31}
    \subset\pi_{31}^7\circ2\sigma_{31}=0.
  \end{equation*}
  By \eqref{nu7si10} and \eqref{w11},
  \begin{equation*}
    \nu_7\sigma_{10}^4=\sigma'\nu_{14}\sigma_{17}^3=0.
  \end{equation*}
  Since $\nu_{16}^\ast\in\{\sigma_{16},2\sigma_{23},\nu_{30}\}_1$ (see
  \cite[p.~153]{Tod62}) and
  \begin{equation*}
    \{\nu_{13},\sigma_{16},2\sigma_{23}\}=
    \xi_{13}+y(\lambda+2\xi_{13})
  \end{equation*}
  for some odd integer $y$ (see \cite[I-Proposition 3.4(8)]{Oda79}),
  relations \eqref{w15}, \eqref{w11} and
  $\xi_{12}\nu_{30}=\sigma_{12}^3$ (see
  \cite[II-Proposition 2.1(2)]{Oda79}) yield
  \begin{align*}
    \nu_{13}\nu_{16}^\ast
    & \in\nu_{13}\circ\{\sigma_{16},2\sigma_{23},\nu_{30}\}\\
    & =-\{\nu_{13},\sigma_{16},2\sigma_{23}\}\circ\nu_{31}\\
    & =(-\xi_{13}-y(\lambda+2\xi_{13}))\nu_{31}
      =\sigma_{13}^3-y\lambda\nu_{31},
  \end{align*}
  and hence, by \eqref{nu7si10},
  \eqref{w11}, $\sigma'E\lambda\equiv0\bmod4E^2\phi''$ (see
  \cite[II-(6.4)]{Oda79}) and $2\nu_{32}^2=0$ (see
  \cite[Proposition 5.11]{Tod62}), we obtain
  \begin{multline*}
    \nu_7\sigma_{10}\nu_{17}^\ast\nu_{35}
    =\sigma'\nu_{14}\nu_{17}^\ast\nu_{35}
    =\sigma'(\sigma_{14}^3-y(E\lambda)\nu_{32})\nu_{35}
    =\sigma'(E\lambda)\nu_{32}^2
    \equiv0\\
    \bmod4E^2\phi''\circ\nu_{32}^2=E^2\phi''\circ4\nu_{32}^2=0.
  \end{multline*}
  By \eqref{nu6ep9},
  \begin{equation*}
    \nu_7\bar\nu_{10}\bar\kappa_{18}
    =\nu_7\varepsilon_{10}\bar\kappa_{18}=0.
  \end{equation*}
  Therefore, every element which is composite of generators of
  $\pi_{38}^{10}$ and $\nu_7$ is $0$.
\end{proof}

\begin{lem}
  \label{l_ncs}
  $\{\nu_7,\chi,2\sigma_{30}\}=0$ for
  $\chi\in\{\eta_{10},\zeta_{11},\sigma_{22}\}$.
\end{lem}
\begin{proof}
  The indeterminacy of $\{\nu_7,\chi,2\sigma_{30}\}$ is trivial by
  Lemma \ref{l_ind}. We may take $\chi=x\beta'$ for some integer $x$
  from the proof of Lemma \ref{l_ezs}. By \cite[Theorem 3(c)]{Oda79}
  and \eqref{b'si30}, we have
  \begin{equation*}
    \{\nu_7,x\beta',2\sigma_{30}\}
    =\{\nu_7,x\beta',\sigma_{30}\}\circ2\iota_{38}
    \in2\pi_{38}^7=\{2\alpha_3'''\},
  \end{equation*}
  and hence $\{\nu_7,x\beta',2\sigma_{30}\}=2y\alpha_3'''$ for some
  integer $y$. By \eqref{b'} and $E^2\theta'=0$ (see
  \cite[p.~80]{Tod62}), we have
  \begin{equation*}
    2yE^5\alpha_3'''=E^5\{\nu_7,x\beta',2\sigma_{30}\}
    \in-\{\nu_{12},0,2\sigma_{35}\}
    =\nu_{12}\circ\pi_{43}^{15}+\pi_{36}^{12}\circ2\sigma_{36}.
  \end{equation*}
  The Toda bracket $-\{\nu_{12},0,2\sigma_{35}\}$ vanishes because
  $\pi_{36}^{12}\cong(\mathbb Z_2)^4$ (see
  \cite[Theorem 1.1(b)]{MMO75}),
  $\nu_{12}\circ\pi_{43}^{15}
  =\{\nu_{12}\varepsilon_{15}\bar\kappa_{23}\}$
  (see \cite[Theorem 2(b)]{Oda79}), and $\nu_{12}\varepsilon_{15}=0$
  by \eqref{nu6ep9}. On the other hand,
  $2yE^5\alpha_3'''\equiv4yE^3\alpha_3^{IV}
  \bmod E^3\pi_{33}^9\circ\sigma_{36}$
  (see \cite[III-Proposition 3.4]{Oda79}). Here, the group
  $E^3\pi_{33}^9\circ\sigma_{36}$ has three generators
  $\delta_{12}\sigma_{36}$, $\bar\mu_{12}\sigma_{29}^2$ and
  $\bar\sigma_{12}'\sigma_{36}$ (see \cite[pp.~34--35]{MMO75}), each
  of which is $0$ by \cite[III-Propositions 2.6(1), 2.6(5)]{Oda79}:
  \begin{equation}
    \label{bm3si20si27}
    \bar\mu_3\sigma_{20}^2=0,
    \quad
    \bar\sigma_7'\sigma_{31}=0
    \quad\text{and}\quad
    \delta_{11}\sigma_{35}=0.
  \end{equation}
  Therefore, we have $4yE^3\alpha_3^{IV}=0$. This yields
  \begin{equation*}
    \{\nu_7,x\beta',2\sigma_{30}\}=2y\alpha_3'''=0,
  \end{equation*}
  since $E^3\alpha_3^{IV}$ and $\alpha_3'''$ are of order $16$ and $8$
  (see \cite[Theorem 3(c)]{Oda79}), respectively.
\end{proof}

\begin{lem}
  \label{l_css}
  $\{\chi,\sigma_{23},2\sigma_{30}\}=\sigma'\eta_{14}\phi_{15}$ for
  $\chi\in\{\nu_7,\eta_{10},\zeta_{11}\}$.
\end{lem}
\begin{proof}
  By Lemma \ref{l_nez},
  $\chi=\sigma'\eta_{14}\varepsilon_{15}+x\eta_7\bar\varepsilon_8$
  for some integer $x$. By \cite[Theorem 2(2)]{Oda85}, we have
  $\sigma'\eta_{14}\varepsilon_{15}=
  \sigma'\eta_{14}^2\sigma_{16}+\eta_7\bar\varepsilon_8$.
  Then
  $\chi=\sigma'\eta_{14}^2\sigma_{16}+(x+1)\eta_7\bar\varepsilon_8$
  and
  \begin{equation*}
    \{\chi,\sigma_{23},2\sigma_{30}\}\subset
    \{\sigma'\eta_{14}^2\sigma_{16},\sigma_{23},2\sigma_{30}\}+
    \{(x+1)\eta_7\bar\varepsilon_8,\sigma_{23},2\sigma_{30}\}.
  \end{equation*}
  By \eqref{ep3ep11} and \eqref{w11},
  \begin{multline*}
    \{\eta_7\bar\varepsilon_8,\sigma_{23},2\sigma_{30}\}
    =\{\nu_7\sigma_{10}\nu_{17}^2,\sigma_{23},2\sigma_{30}\}
    \supset\nu_7
    \circ\{\sigma_{10}\nu_{17}^2,\sigma_{23},2\sigma_{30}\}\\
    \bmod\nu_7\sigma_{10}\nu_{17}^2\circ\pi_{38}^{23}
    +\pi_{31}^7\circ2\sigma_{31}.
  \end{multline*}
  By Lemma \ref{l_ind}, we have
  $\nu_7\circ\{\sigma_{10}\nu_{17}^2,\sigma_{23},2\sigma_{30}\}
  \subset\nu_7\circ\pi_{38}^{10}=0$
  and
  \begin{equation*}
    \nu_7\sigma_{10}\nu_{17}^2\circ\pi_{38}^{23}
    +\pi_{31}^7\circ2\sigma_{31}
    \subset\nu_7\circ\pi_{38}^{10}+\pi_{31}^7\circ2\sigma_{31}=0.
  \end{equation*}
  These give
  $\{(x+1)\eta_7\bar\varepsilon_8,\sigma_{23},2\sigma_{30}\}=0$. On
  the other hand, we have
  \begin{multline*}
    \{\sigma'\eta_{14}^2\sigma_{16},\sigma_{23},2\sigma_{30}\}
    \supset
    \sigma'\eta_{14}\circ
    \{\eta_{15}\sigma_{16},\sigma_{23},2\sigma_{30}\}
    \ni
    \sigma'\eta_{14}\phi_{15}\\
    \bmod
    \sigma'\eta_{14}^2\sigma_{16}\circ\pi_{38}^{23}+
    \pi_{31}^7\circ2\sigma_{31}
  \end{multline*}
  since $\{\eta_9\sigma_{10},\sigma_{17},2\sigma_{24}\}\ni\phi_9$ (see
  \cite[I-Proposition 3.4(5)]{Oda79}). Lemma \ref{l_ind},
  \eqref{ep3ep11} and the above equation
  $\sigma'\eta_{14}\varepsilon_{15}=
  \sigma'\eta_{14}^2\sigma_{16}+\eta_7\bar\varepsilon_8$
  yield
  \begin{equation*}
    \sigma'\eta_{14}^2\sigma_{16}\circ\pi_{38}^{23}
    +\pi_{31}^7\circ2\sigma_{31}
    =(\sigma'\eta_{14}\varepsilon_{15}-\nu_7\sigma_{10}\nu_{17}^2)
    \circ\pi_{38}^{23}
    =\sigma'\eta_{14}\varepsilon_{15}\circ\pi_{38}^{23}.
  \end{equation*}
  Then the indeterminacy is trivial because
  $\pi_{38}^{23}=\{\rho_{23},\bar\varepsilon_{23}\}$ (see
  \cite[Theorem 10.10]{Tod62}), $\varepsilon_5\rho_{13}=0$ (see
  \cite[I-Proposition 3.1(4)]{Oda79}), and
  $\eta_3\varepsilon_4\bar\varepsilon_{12}=
  \varepsilon_3\eta_{11}\bar\varepsilon_{12}=0$
  by \eqref{et6ep7} and \eqref{ep9ep17}. Therefore, we have
  $\{\sigma'\eta_{14}^2\sigma_{16},\sigma_{23},2\sigma_{30}\}=
  \sigma'\eta_{14}\phi_{15}$
  and this completes the proof.
\end{proof}

Lemmas \ref{l_ncs} and \ref{l_css} simplify \eqref{e_jac1} as follows.
\begin{equation}
  \label{e_jac2}
  \sigma'\eta_{14}\phi_{15}
  \in\{\nu_7,\eta_{10},\{\zeta_{11},\sigma_{22},2\sigma_{29}\}\}.
\end{equation}
Next, we investigate the Toda bracket
$\{\nu_7,\eta_{10},\{\zeta_{11},\sigma_{22},2\sigma_{29}\}\}$.
The following lemma extends \cite[Lemma 3.5]{HM00}.

\begin{lem}
  \label{l_zss}
  $\{\zeta_{11},\sigma_{22},2\sigma_{29}\}
  =\nu_{11}^2\bar\kappa_{17}
  \pm2\tau'''+x\sigma_{11}\bar\sigma_{18}$
  for $x=0$ or $1$.
\end{lem}
\begin{proof}
  We consider an inclusion
  \begin{equation*}
    \{\zeta_{11},2\sigma_{22},\sigma_{29}\}
    \supset\{\zeta_{11},\sigma_{22},2\sigma_{29}\}
    \bmod\zeta_{11}\circ\pi_{37}^{22}+\pi_{30}^{11}\circ\sigma_{30}.
  \end{equation*}
  The groups $\pi_{37}^{22}=\{\rho_{22},\bar\varepsilon_{22}\}$ and
  $\pi_{30}^{11}=\{\lambda'\eta_{29},\xi'\eta_{29},
  \bar\sigma_{11},\bar\zeta_{11}\}$
  are obtained in \cite[Theorems 10.10, 12.23]{Tod62}. Since
  $\zeta_{11}\rho_{22}\equiv0\bmod8\pi_{37}^{11}$ (see
  \cite[I-Proposition 3.5(7)]{Oda79}) and $\pi_{37}^{11}\cong\mathbb
  Z_8\oplus(\mathbb Z_2)^4$ (see \cite[Theorem 1(b)]{Oda79}), we have
  $\zeta_{11}\rho_{22}=0$. We also have
  $\zeta_{11}\bar\varepsilon_{22}=\zeta_{11}\eta_{22}\kappa_{23}=0$ by
  \eqref{et6ka7} and \eqref{et5ze6}. Thus
  $\zeta_{11}\circ\pi_{37}^{22}=0$. Moreover, we obtain
  $\pi_{30}^{11}\circ\sigma_{30}=0$:
  $\lambda'\eta_{29}\sigma_{30}=\xi'\eta_{29}\sigma_{30}=0$ by
  \eqref{pret29si30}; $\bar\sigma_{10}\sigma_{29}=0$ by
  \cite[I-Proposition 6.4(6)]{Oda79}; and
  $\bar\zeta_{11}\sigma_{30}=0$ by \cite[I-(8.14)]{Oda79}. This
  implies that $\{\zeta_{11},\sigma_{22},2\sigma_{29}\}$ equals
  $\{\zeta_{11},2\sigma_{22},\sigma_{29}\}$ and consists of one
  element. Therefore, we examine the Toda bracket
  $\{\zeta_{11},2\sigma_{22},\sigma_{29}\}$ below.
  By \cite[the proof of I-(8.22)]{Oda79}, we have
  $H\{\zeta_{10},2\sigma_{21},\sigma_{28}\}_1=H\tau''$ and hence
  \begin{equation*}
    \{\zeta_{10},2\sigma_{21},\sigma_{28}\}_1
    \subset\tau''+\ker H=\tau''
    +\{2\tau'',\sigma_{10}\bar\sigma_{17},\bar\kappa_{10}\nu_{30}^2,
    \nu_{10}^2\bar\kappa_{16},\eta_{10}\mu_{3,11}\}
  \end{equation*}
  by \cite[I-(8.12), (8.13)]{Oda79}.
  Then a relation
  $\bar\kappa_{11}\nu_{31}^2=\nu_{11}^2\bar\kappa_{17}$ (see
  \cite[I-Proposition 6.4(7)]{Oda79}) shows
  \begin{align*}
    \{\zeta_{11},2\sigma_{22},\sigma_{29}\}
    & =-E\{\zeta_{10},2\sigma_{21},\sigma_{28}\}_1\\
    & \in-E\tau''
      +\{2E\tau'',\sigma_{11}\bar\sigma_{18},
      \nu_{11}^2\bar\kappa_{17},\eta_{11}\mu_{3,12}\}.
  \end{align*}
  By \cite[I-(8.18), Theorem 1(b)]{Oda79}, $2\tau'''=-E\tau''$,
  $\tau'''$ is of order $8$, and $\sigma_{11}\bar\sigma_{18}$,
  $\nu_{11}^2\bar\kappa_{17}$, $\eta_{11}\mu_{3,12}$ are of order $2$.
  Then $\{\zeta_{11},2\sigma_{22},\sigma_{29}\}$ is written as
  \begin{equation*}
    \{\zeta_{11},2\sigma_{22},\sigma_{29}\}=
    \pm2\tau'''+a\sigma_{11}\bar\sigma_{18}+
    b\nu_{11}^2\bar\kappa_{17}+c\eta_{11}\mu_{3,12}
  \end{equation*}
  with coefficients $a,b,c$, each of which is $0$ or $1$. Since
  $\sigma_{14}\bar\sigma_{21}=0$ and
  $E^8(2\tau''')=-E^9\tau''=4E^7\tau^{IV}=0$ (see
  \cite[I-Proposition 6.4(10), (8.22), (8.24)]{Oda79}), we have
  \begin{equation*}
    \langle\zeta,2\sigma,\sigma\rangle=
    b\nu^2\bar\kappa+c\eta\mu_{3,\ast}\in\pi_{26}^s.
  \end{equation*}
  Here,
  $\pi_{26}^s=\mathbb Z_2\{\nu^2\bar\kappa\}\oplus
  \mathbb Z_2\{\eta\mu_{3,*}\}$
  by \cite[Theorem 1(b)]{Oda79} and
  $\langle\zeta,2\sigma,\sigma\rangle=
  \langle\sigma,2\sigma,\zeta\rangle=\nu^2\bar\kappa$
  by \cite[Theorem 1]{HM00}. Thus $b=1$ and $c=0$: the proof is
  completed.
\end{proof}

We recall the definition and a property of
$\bar{\bar\sigma}_6'\in\pi_{30}^6$ from \cite[(3.8)]{MMO75}.
\begin{equation}
  \label{bbs6}
  \bar{\bar\sigma}_6'\in\{\nu_6,\eta_9,\bar\sigma_{10}\}_3,
  \quad2\bar{\bar\sigma}_6'=0.
\end{equation}

\begin{lem}
  \label{l_nec}
  $\{\nu_7,\eta_{10},\chi\}
  \equiv\bar\nu_7\nu_{15}\bar\kappa_{18}
  \bmod\sigma'\bar{\bar\sigma}_{14}'$
  for $\chi\in\{\zeta_{11},\sigma_{22},2\sigma_{29}\}$.
\end{lem}
\begin{proof}
  By Lemma \ref{l_zss}, $\chi$ has the form
  $\nu_{11}^2\bar\kappa_{17}\pm2\tau'''+x\sigma_{11}\bar\sigma_{18}$
  for $x=0$ or $1$. Then there is an inclusion
  \begin{equation*}
    \{\nu_7,\eta_{10},\chi\}
    \subset\{\nu_7,\eta_{10},\nu_{11}^2\bar\kappa_{17}\}
    +\{\nu_7,\eta_{10},2\tau'''\}
    +\{\nu_7,\eta_{10},x\sigma_{11}\bar\sigma_{18}\}.
  \end{equation*}
  Notice that $\nu_7\circ\pi_{38}^{10}=0$, $\pi_{12}^7=0$ and
  $\pi_{19}^7=0$ by Lemma \ref{l_ind} and \cite[Proposition 5.9,
  Theorem 7.6]{Tod62}. Since
  \begin{equation}
    \label{bnu6}
    \bar\nu_6\in\{\nu_6,\eta_9,\nu_{10}\},
    \quad\text{\cite[p.~53]{Tod62},}
  \end{equation}
  we have
  \begin{multline*}
    \bar\nu_7\nu_{15}\bar\kappa_{18}
    \in\{\nu_7,\eta_{10},\nu_{11}\}\circ\nu_{15}\bar\kappa_{18}
    \subset\{\nu_7,\eta_{10},\nu_{11}^2\bar\kappa_{17}\}\\
    \bmod\nu_7\circ\pi_{38}^{10}
    +\pi_{12}^7\circ\nu_{12}^2\bar\kappa_{18}=0.
  \end{multline*}
  By \eqref{mono}, a relation
  $2\iota_{11}\circ\tau'''=2\tau'''$ holds, and hence
  \begin{multline*}
    \{\nu_7,\eta_{10},2\tau'''\}
    \supset\{\nu_7,\eta_{10},2\iota_{11}\}\circ E\tau'''
    \subset\pi_{12}^7\circ E\tau'''=0\\
    \bmod\nu_7\circ\pi_{38}^{10}+\pi_{12}^7\circ2E\tau'''=0.
  \end{multline*}
  By \eqref{w9}, \eqref{nu7si10}, \eqref{et5bs6} and \eqref{bbs6},
  \begin{align*}
    \{\nu_7,\eta_{10},\sigma_{11}\bar\sigma_{18}\}
    & \subset\{\nu_7,\eta_{10}\sigma_{11},\bar\sigma_{18}\}\\
    & \supset\{\nu_7\sigma_{10},\eta_{17},\bar\sigma_{18}\}\\
    & \supset y\sigma'\circ\{\nu_{14},\eta_{17},\bar\sigma_{18}\}\\
    & \ni\sigma'\bar{\bar\sigma}_{14}'
      \bmod\nu_7\circ\pi_{38}^{10}+\pi_{19}^7\circ\bar\sigma_{19}=0
  \end{align*}
  for some odd integer $y$. These imply
  $\{\nu_7,\eta_{10},\nu_{11}^2\bar\kappa_{17}\}
  =\bar\nu_7\nu_{15}\bar\kappa_{18}$,
  $\{\nu_7,\eta_{10},2\tau'''\}=0$ and
  $\{\nu_7,\eta_{10},\sigma_{11}\bar\sigma_{18}\}
  =\sigma'\bar{\bar\sigma}_{14}'$.
\end{proof}

From Lemma \ref{l_nec} and \eqref{e_jac2}, we have
\begin{equation*}
  \sigma'\eta_{14}\phi_{15}
  \equiv\bar\nu_7\nu_{15}\bar\kappa_{18}
  \bmod\sigma'\bar{\bar\sigma}_{14}'.
\end{equation*}
Since
$\eta_4\phi_5\equiv\delta_4
\bmod\bar\mu_4\sigma_{21},(E\varepsilon')\kappa_{14}$
(see \cite[I-Proposition 3.5(9)]{Oda79}),
\begin{equation*}
  \sigma'\delta_{14}
  \equiv\bar\nu_7\nu_{15}\bar\kappa_{18}
  \bmod\sigma'\bar{\bar\sigma}_{14}',
  \sigma'\bar\mu_{14}\sigma_{31},
  \sigma'(E^{11}\varepsilon')\kappa_{24}.
\end{equation*}
Here, $E^{11}\varepsilon'=\pm2\nu_{14}\sigma_{17}=0$ by
\cite[(7.10)]{Tod62} and \eqref{w11}. Furthermore, we obtain
$\bar\sigma_6'\equiv\bar{\bar\sigma}_6'
\bmod P(\xi_{13})\eta_{29},\nu_6\sigma_9\kappa_{16}$
by \cite[Lemma 4.5]{MMpre}.
Since $\nu_8\sigma_{11}\kappa_{18}=P(\nu_{17}\kappa_{20})$ (see
\cite[(5.11)]{MMO75}), we see that
$\bar\sigma_9'=\bar{\bar\sigma}_9'$. This completes the proof of
Theorem \ref{main1}.

\section{Proof of Theorem \ref{IM}\ref{IM_p4110}.}

According to \cite[the proof of III-Proposition 3.3(2)]{Oda79},
\begin{equation}
  \label{e_tar}
  \{\kappa_{10}+8x\sigma_{10}^2,2\iota_{24},\eta_{24}\}\ni0,
  \quad\{2\iota_{23},\eta_{23},\sigma_{24}^2\}=0
\end{equation}
and $\kappa_{10}^\ast$ is a representative of a tertiary composition
\begin{equation*}
  \{\kappa_{10}+8x\sigma_{10}^2,
  2\iota_{24},\eta_{24},\sigma_{25}^2\}_1,
\end{equation*}
where $x$ is an integer satisfying that
$\{\kappa_{11},2\iota_{25},\eta_{25}\}=x\sigma_{11}\mu_{18}$.
It holds that \cite[III-(10.5)]{Oda79}:
\begin{equation*}
  2\kappa_{10}^\ast
  \in\{\bar\varepsilon_{10},\eta_{25},\sigma_{26}^2\}
  \bmod P(\nu_{21}\bar\sigma_{24})
  \quad\text{and}\quad
  2\kappa_{10}^\ast
  =x\delta_{10}\sigma_{34}
  \text{ for }x=0\text{ or }1,
\end{equation*}
where
$\delta_{10}\sigma_{34}=P(\nu_{21}\bar\sigma_{24})=
[\iota_{10},\nu_{10}\bar\sigma_{13}]$
by \eqref{imageP} and \eqref{imageW}.

We denote by $M^n=S^{n-1}\cup_{2\iota_{n-1}}e^n$ the
$\mathbb Z_2$-Moore space, $i_n\colon S^{n-1}\rightarrow M^n$ the
inclusion map, and $p_n\colon M^n\rightarrow S^n$ the collapsing map.
Let $\tilde\eta_n\in\pi_{n+2}(M^{n+1})$ be a coextension of $\eta_n$
for $n\geq3$. Notice that
$\tilde\eta_n\in\pi_{n+2}(M^{n+1})\cong\mathbb Z_4$ is a generator
satisfying
\begin{equation}
  \label{coet}
  p_{n+1}\tilde\eta_n=\eta_{n+1},
  \quad2\tilde\eta_n=i_{n+1}\eta_n^2
  \quad\text{and}
  \quad\tilde\eta_n\in\{i_{n+1},2\iota_n,\eta_n\}.
\end{equation}

By \eqref{w2} and \eqref{et9si10si17}, a Toda bracket
$\{2\iota_{11},\eta_{11},\sigma_{12}^2\}\subset\pi_{27}^{11}$ is well
defined. Since $\pi_{27}^{11}=\mathbb Z_2\{\sigma_{11}\mu_{18}\}$,
$\pi_{39}^{23}=
\mathbb Z_2\{\omega_{23}\}\oplus\mathbb Z_2\{\sigma_{23}\mu_{30}\}$
(see \cite[Theorem 12.16]{Tod62}) and
$E^{12}\colon\pi_{27}^{11}\rightarrow\pi_{39}^{23}$ is a monomorphism,
\eqref{e_tar} implies $\{2\iota_{11},\eta_{11},\sigma_{12}^2\}=0$.
Then \eqref{coet} gives
\begin{equation*}
  \tilde\eta_{11}\sigma_{13}^2
  \in\{i_{12},2\iota_{11},\eta_{11}\}\circ\sigma_{13}^2
  =-(i_{12}\circ\{2\iota_{11},\eta_{11},\sigma_{12}^2\})
  =0.
\end{equation*}
By \eqref{e_tar}, there exists an extension
$\overline{\kappa_{10}+8x\sigma_{10}^2}\in[M^{25},S^{10}]$ of
$\kappa_{10}+8x\sigma_{10}^2$ such that
$\overline{\kappa_{10}+8x\sigma_{10}^2}\circ\tilde\eta_{24}=0$. Thus,
we may define a Toda bracket
\begin{equation*}
  \{\overline{\kappa_{10}+8x\sigma_{10}^2},
  \tilde\eta_{24},\sigma_{26}^2\}_{13}.
\end{equation*}
Hereafter, $\kappa_{10}^\ast$ is chosen as a representative of this
Toda bracket.

\begin{lem}
  \label{l_tb1}
  $\kappa_{10}^\ast
  =\{\overline{\kappa_{10}+8x\sigma_{10}^2},\tilde\eta_{24},
  \sigma_{26}^2\}_{13}$
  and
  $2\kappa_{10}^\ast
  =E\{\bar\varepsilon_9,\eta_{24},\sigma_{25}^2\}_{12}$.
\end{lem}
\begin{proof}
  Notice that $\{\bar\varepsilon_9,\eta_{24},\sigma_{25}^2\}_{12}$ is
  well defined by \eqref{ep3ep11}, \eqref{et9si10si17} and
  \eqref{ep9ep17}. The indeterminacy of
  $\{\overline{\kappa_{10}+8x\sigma_{10}^2},
  \tilde\eta_{24},\sigma_{26}^2\}_{13}$
  is
  \begin{equation}
    \label{e_ind}
    \overline{\kappa_{10}+8x\sigma_{10}^2}\circ E^{13}\pi_{28}(M^{12})
    +\pi_{27}^{10}\circ\sigma_{27}^2.
  \end{equation}
  If the group is trivial, the first half is proved. Moreover, the
  triviality shows the second half since \eqref{coet}, \eqref{w2} and
  \eqref{et6ka7} give
  \begin{align*}
    2\kappa_{10}^\ast
    & \in\{\overline{\kappa_{10}+8x\sigma_{10}^2},\tilde\eta_{24},
      \sigma_{26}^2\}_{13}\circ2\iota_{41}\\
    & \subset\{\overline{\kappa_{10}+8x\sigma_{10}^2},
      2\tilde\eta_{24},\sigma_{26}^2\}_{13}\\
    & =\{\overline{\kappa_{10}+8x\sigma_{10}^2},
      i_{25}\eta_{24}^2,\sigma_{26}^2\}_{13}\\
    & \supset\{\kappa_{10}\eta_{24},\eta_{25},\sigma_{26}^2\}_{13}\\
    & =\{\bar\varepsilon_{10},\eta_{25},\sigma_{26}^2\}_{13}\\
    & \supset E\{\bar\varepsilon_9,\eta_{24},\sigma_{25}^2\}_{12}\\
    & \bmod\overline{\kappa_{10}+8x\sigma_{10}^2}
      \circ E^{13}\pi_{28}(M^{12})+\pi_{27}^{10}\circ\sigma_{27}^2.
  \end{align*}
  Therefore, we shall prove the group \eqref{e_ind} to be $0$.

  By \cite[Theorem 12.17]{Tod62}, we have
  $\pi_{27}^{10}=
  \{\sigma_{10}\eta_{17}\mu_{18},\nu_{10}\kappa_{13},\bar\mu_{10}\}$.
  Then $\pi_{27}^{10}\circ\sigma_{27}^2=0$ because
  $\mu_3\sigma_{12}^2=0$ by \cite[(2.9)]{MMO75},
  \begin{equation}
    \label{ka7si21}
    \kappa_7\sigma_{21}=0
  \end{equation}
  by \cite[II-Proposition 2.1(2)]{Oda79}, and
  $\bar\mu_3\sigma_{20}^2=0$ by \eqref{bm3si20si27}.

  Let $\varphi_n\in\pi_n(M^n,S^{n-1})$ be the characteristic map of
  the $n$-cell of $M^n$ and let
  $j\colon(M^n,\ast)\rightarrow(M^n,S^{n-1})$ be the inclusion. By
  \cite[Theorem 12.16]{Tod62}, we know that
  $\pi_{n+16}(S^n)=\mathbb Z_2\{\sigma_n\mu_{n+7}\}$ for $n=11,12$.
  Then, by \cite[(2.7)]{Jam54}, there is a split exact sequence
  \begin{equation*}
    0\rightarrow
    Q\pi_{17}(S^{11})\hookrightarrow
    \pi_{28}(M^{12},S^{11})\xrightarrow{p_{12\ast}}
    \pi_{28}(S^{12})\rightarrow
    0,
  \end{equation*}
  where $Q\colon\pi_{17}(S^{11})\rightarrow\pi_{28}(M^{12},S^{11})$
  assigns to $x\in\pi_{17}(S^{11})$ the relative Whitehead product
  $[\varphi_{12},x]$, and the kernel of $Q$ is
  $(2\iota_{11\ast}\circ E^{-12}\circ H)\pi_{29}(S^{12})
  \subset2\pi_{17}(S^{11})$.
  Since $\pi_{17}(S^{11})=\mathbb Z_2\{\nu_{11}^2\}$ (see
  \cite[Proposition 5.11]{Tod62}), the kernel of $Q$ is trivial, and
  hence
  $\pi_{28}(M^{12},S^{11})=
  \mathbb Z_2\{[\varphi_{12},\nu_{11}^2]\}\oplus\mathbb Z_2\{\chi\}$
  for an element $\chi\in p_{12\ast}^{-1}(\sigma_{12}\mu_{19})$.
  Notice that
  $\partial[\varphi_{12},\nu_{11}^2]=[2\iota_{11},\nu_{11}^2]=
  [\iota_{11},2\nu_{11}^2]=0$
  and $\partial\chi=2\sigma_{11}\mu_{18}=0$, where $\partial$ is the
  connecting homomorphism
  $\pi_{28}(M^{12},S^{11})\rightarrow\pi_{27}(S^{11})$. So, by making
  use of the homotopy exact sequence of a pair $(M^{12},S^{11})$, we
  obtain
  \begin{equation*}
    \pi_{28}(M^{12})=\{\chi_1,\chi_2\}
    +i_{12\ast}\pi_{28}(S^{11}),
  \end{equation*}
  where
  $\chi_1\in j_\ast^{-1}\chi\subset
  j_\ast^{-1}p_{12\ast}^{-1}(\sigma_{12}\mu_{19})$
  and $\chi_2\in j_\ast^{-1}[\varphi_{12},\nu_{11}^2]$. We have
  \begin{equation*}
    p_{12\ast}j_\ast(\tilde\eta_{11}\rho_{13})=
    p_{12}\tilde\eta_{11}\rho_{13}=
    \eta_{12}\rho_{13}=
    \sigma_{12}\mu_{19}
  \end{equation*}
  because $\sigma_{12}\mu_{19}=\eta_{12}\rho_{13}$ (see
  \cite[Proposition 12.20]{Tod62}) and by \eqref{coet}. Thus, we may
  take $\chi_1=\tilde\eta_{11}\rho_{13}$. Since
  $j_\ast E\chi_2=E'[\varphi_{12},\nu_{11}^2]=0$ for the relative
  suspension
  $E'\colon\pi_{28}(M^{12},S^{11})\rightarrow\pi_{29}(M^{13},S^{12})$
  (see \cite{Tod52}), $E\chi_2$ is contained in the kernel of
  $j_\ast\colon\pi_{29}(M^{13})\rightarrow\pi_{29}(M^{13},S^{12})$,
  that is, $E\chi_2\in i_{13\ast}\pi_{29}(S^{12})$. Therefore,
  \eqref{e_ind} is a subgroup of
  \begin{equation*}
    \{\overline{\kappa_{10}+8x\sigma_{10}^2}
    \circ\tilde\eta_{24}\rho_{26}\}
    +\overline{\kappa_{10}+8x\sigma_{10}^2}
    \circ E^{12}(i_{13\ast}\pi_{29}(S^{12})).
  \end{equation*}
  From the definition of $\overline{\kappa_{10}+8x\sigma_{10}^2}$, we
  have
  $\overline{\kappa_{10}+8x\sigma_{10}^2}
  \circ\tilde\eta_{24}\rho_{26}=0$.
  By the fact that
  $E^{12}\pi_{29}(S^{12})=
  \mathbb Z_2\{\varepsilon_{24}^\ast\}\oplus
  \mathbb Z_2\{\sigma_{24}\eta_{31}\mu_{32}\}\oplus
  \mathbb Z_2\{\nu_{24}\kappa_{27}\}\oplus
  \mathbb Z_2\{\bar\mu_{24}\}$
  (see \cite[Theorem 12.17]{Tod62}),
  \begin{align*}
    \overline{\kappa_{10}+8x\sigma_{10}^2}
    \circ E^{12}(i_{13\ast}\pi_{29}(S^{12}))
    & =(\kappa_{10}+8x\sigma_{10}^2)\circ E^{12}\pi_{29}(S^{12})\\
    & =\kappa_{10}\circ E^{12}\pi_{29}(S^{12})\\
    & =\kappa_{10}\circ\{\varepsilon_{24}^\ast,
      \sigma_{24}\eta_{31}\mu_{32},
      \nu_{24}\kappa_{27},\bar\mu_{24}\}.
  \end{align*}
  By equations \eqref{et6ka7},
  $\eta_{13}\omega_{14}=\varepsilon_{13}^\ast$ (see
  \cite[Proposition (2.13)(1)]{Ogu64}) and
  \begin{equation}
    \label{be3om18}
    \bar\varepsilon_3\omega_{18}=0,
    \quad\text{\cite[III-(9.5)]{Oda79},}
  \end{equation}
  we have
  $\kappa_{10}\varepsilon_{24}^\ast
  =\kappa_{10}\eta_{24}\omega_{25}
  =\bar\varepsilon_{10}\omega_{25}
  =0$.
  We also have $\kappa_{10}\sigma_{24}\eta_{31}\mu_{32}=0$ by
  \eqref{ka7si21}. A relation $\kappa_7\nu_{21}=\nu_7\kappa_{10}$ (see
  \cite[Proposition (2.13)(2)]{Ogu64}) and Lemma \ref{l_ind} give
  $\kappa_{10}\nu_{24}\kappa_{27}=
  \nu_{10}\kappa_{13}^2\in
  E^3(\nu_7\circ\pi_{38}^{10})=0$.
  By \cite[Proposition 3.1]{Tod62} and Lemma \ref{bm3ka20},
  $\kappa_{10}\bar{\mu}_{24}=\bar{\mu}_{10}\kappa_{27}=0$ is obtained.
  Therefore, \eqref{e_ind} is trivial.
\end{proof}

We know $\bar\nu_{17}=\{\nu_{17},\eta_{20},\nu_{21}\}$,
$\pi_{22}^9=\mathbb Z_2\{\sigma_9\nu_{16}^2\}$ and
$\sigma_9\nu_{16}^2\circ\nu_{22}=\sigma_9\nu_{16}^3\neq0$ (see
\cite[Lemma 6.2, Theorems 7.7, 12.6]{Tod62}). A Toda bracket
$\{\varepsilon_9,\nu_{17},\eta_{20}\}\subset\pi_{22}^9$ is well
defined by \eqref{ep4nu12} and \eqref{w5}. Since
$\{\varepsilon_9,\nu_{17},\eta_{20}\}\circ\nu_{22}=
-(\varepsilon_9\circ\{\nu_{17},\eta_{20},\nu_{21}\})=
\varepsilon_9\bar{\nu}_{17}$ and
\begin{equation}
  \label{ep9bn17}
  \varepsilon_9\bar{\nu}_{17}=0,
\end{equation}
by \eqref{w11} and \eqref{ep3ep11}, we have
\begin{equation}
  \label{e_ene}
  \{\varepsilon_9,\nu_{17},\eta_{20}\}=0.
\end{equation}

\begin{lem}
  \label{l_tb2}
  $\{\varepsilon_9,\bar\nu_{17},\sigma_{25}^2\}_9=0$
  and
  $\{\varepsilon_9,\varepsilon_{17},\sigma_{25}^2\}_9
  =\{\bar\varepsilon_9,\eta_{24},\sigma_{25}^2\}_9$
  consists of one element.
\end{lem}
\begin{proof}
  By \eqref{ep3si11}, \eqref{ep9ep17} and \eqref{ep9bn17},
  $\{\varepsilon_9,\bar\nu_{17},\sigma_{25}^2\}_9$ and
  $\{\varepsilon_9,\varepsilon_{17},\sigma_{25}^2\}_9$ are well
  defined. Well-definedness of the remaining Toda bracket
  $\{\bar\varepsilon_9,\eta_{24},\sigma_{25}^2\}_9$ is stated in the
  proof of Lemma \ref{l_tb1}. The indeterminacy of
  $\{\varepsilon_9,\bar\nu_{17},\sigma_{25}^2\}_9$ is
  $\varepsilon_9\circ E^9\pi_{31}^{8}+\pi_{26}^9\circ\sigma_{26}^2$,
  which is the same as that of
  $\{\varepsilon_9,\varepsilon_{17},\sigma_{25}^2\}_9$. Here,
  \begin{equation*}
    E^9\pi_{31}^{8}
    =\{2\bar\rho_{17},\nu_{17}\bar\kappa_{20},\phi_{17}\}
  \end{equation*}
  and
  \begin{equation*}
    \pi_{26}^9=\{\sigma_9\eta_{16}\mu_{17},
    \nu_9\kappa_{12},\bar\mu_9,\eta_9\mu_{10}\sigma_{19}\}
  \end{equation*}
  by \cite[pp.~24--27]{MMO75} and \cite[Theorem 12.7]{Tod62},
  respectively.
  The group $\varepsilon_9\circ E^9\pi_{31}^{8}$ is trivial:
  $\varepsilon_9\circ2\bar\rho_{17}=
  2\varepsilon_9\circ\bar\rho_{17}=0$
  as $\varepsilon_9$ is of order $2$ (see \cite[Theorem 7.1]{Tod62});
  $\varepsilon_9\nu_{17}=0$ by \eqref{ep4nu12}; and
  $\varepsilon_9\phi_{17}=0$ by \cite[III-Proposition 2.6(1)]{Oda79}.
  We also have $\pi_{26}^9\circ\sigma_{26}^2=0$ in the same way that
  $\pi_{27}^{10}\circ\sigma_{27}^2=0$ is proved in Lemma \ref{l_tb1}.
  Then the two Toda brackets consist of one element.

  By \eqref{bnu6}, we may set $\bar\nu_6=\chi_1\chi_2$, where $\chi_1$
  is an extension of $\nu_6$ and $\chi_2$ is a coextension of
  $\nu_{10}$ with respect to $\eta_9$. Let $\mathbb C\mathrm P^2$ be
  the complex plane,
  $i_{\mathbb C}\colon S^2\rightarrow\mathbb C\mathrm P^2$ be the
  inclusion and
  $p_{\mathbb C}\colon\mathbb C\mathrm P^2\rightarrow S^4$ be the
  collapsing map. By \eqref{e_ene}, we obtain
  \begin{equation*}
    \varepsilon_9E^{11}\chi_1
    \in\varepsilon_9\circ\{\nu_{17},\eta_{20},E^{17}p_{\mathbb C}\}
    =-(\{\varepsilon_9,\nu_{17},\eta_{20}\}\circ E^{18}p_{\mathbb C})
    =0.
  \end{equation*}
  By the fact that $\{\eta_n,\nu_{n+1},\sigma_{n+4}\}=0$ for $n\geq11$
  (see \cite[Lemma 4.1(2)]{MMpre}),
  \begin{equation*}
    (E^2\chi_2)\sigma_{16}
    \in\{E^9i_{\mathbb C},\eta_{11},\nu_{12}\}\circ\sigma_{16}
    =-(E^9i_{\mathbb C}\circ\{\eta_{11},\nu_{12},\sigma_{15}\})
    =0.
  \end{equation*}
  So, we have
  \begin{align*}
    \{\varepsilon_9,\bar\nu_{17},\sigma_{25}^2\}_9
    & =
    \{\varepsilon_9,E^{11}(\chi_1\chi_2),\sigma_{25}^2\}_9
    \\
    & =
    \{\varepsilon_9E^{11}\chi_1,E^{11}\chi_2,\sigma_{25}^2\}_9
    =\{0,E^{11}\chi_2,\sigma_{25}^2\}_9
    =0.
  \end{align*}
  This leads to the first half.

  Since $\varepsilon_5=\{\nu_5^2,2\iota_{11},\eta_{11}\}$ (see
  \cite[(7.6)]{Tod62}), we may set
  $\varepsilon_5=\chi_3\tilde\eta_{11}$, where $\chi_3$ is an
  extension of $\nu_5^2$ with respect to $2\iota_{11}$. Since
  $\bar\varepsilon_5=\{\varepsilon_5,\nu_{13}^2,2\iota_{19}\}_1$ (see
  \cite[III-Proposition 2.3(5)]{Oda79}), we have
  \begin{equation*}
    \varepsilon_5E^8\chi_3
    \in\varepsilon_5\circ E\{\nu_{12}^2,2\iota_{18},p_{18}\}
    =\{\varepsilon_5,\nu_{13}^2,2\iota_{19}\}_1\circ p_{20}
    =\bar\varepsilon_5p_{20}.
  \end{equation*}
  Therefore,
  \begin{align*}
    \{\varepsilon_9,\varepsilon_{17},\sigma_{25}^2\}_9
    & =
    \{\varepsilon_9,E^{12}(\chi_3\tilde\eta_{11}),\sigma_{25}^2\}_9
    \\
    & =
    \{\varepsilon_9E^{12}\chi_3,\tilde\eta_{23},\sigma_{25}^2\}_9
    \\
    & =\{\bar\varepsilon_9p_{24},\tilde\eta_{23},\sigma_{25}^2\}_9
    \\
    & \in
    \{\bar\varepsilon_9,\eta_{24},\sigma_{25}^2\}_9
    \bmod\bar\varepsilon_9\circ E^9\pi_{31}^{15}
    +\pi_{26}^9\circ\sigma_{26}^2.
  \end{align*}
  Here, $E^9\pi_{31}^{15}=\{\omega_{24},\sigma_{24}\mu_{31}\}$ by
  \cite[Theorem 12.16]{Tod62}, $\bar\varepsilon_9\omega_{24}=0$ by
  \eqref{be3om18}, and $\bar\varepsilon_9\sigma_{24}\mu_{31}=0$ by
  \eqref{be3si18}. Then the indeterminacy is
  $\pi_{26}^9\circ\sigma_{26}^2$, which equals $0$ as above. This
  leads to the second half and completes the proof.
\end{proof}

By \cite[(3.5)]{MMO75}, $\delta_3$ is an element in
$\{\varepsilon_3,\varepsilon_{11}+\bar\nu_{11},\sigma_{19}\}_1$. Since
$E^3\pi_{24}^8=E\pi_{26}^{10}=\{\sigma_{11}\mu_{18}\}$ (see
\cite[Theorems 12.6, 12.16]{Tod62}), the indeterminacy of the Toda
bracket coincides with that of
$\{\varepsilon_3,\varepsilon_{11}+\bar\nu_{11},\sigma_{19}\}_3$. Then
$\delta_3\in
\{\varepsilon_3,\varepsilon_{11}+\bar\nu_{11},\sigma_{19}\}_3$,
and hence, by Lemma \ref{l_tb2},
\begin{align*}
  \delta_9\sigma_{33}
  & \in
  \{\varepsilon_9,\varepsilon_{17}+\bar\nu_{17},\sigma_{25}\}_9
  \circ\sigma_{33}
  \\
  & \subset
  \{\varepsilon_9,\varepsilon_{17}+\bar\nu_{17},\sigma_{25}^2\}_9
  \\
  & =
  \{\varepsilon_9,\varepsilon_{17},\sigma_{25}^2\}_9
  +\{\varepsilon_9,\bar\nu_{17},\sigma_{25}^2\}_9
  \\
  & =
  \{\bar\varepsilon_9,\eta_{24},\sigma_{25}^2\}_9.
\end{align*}
This gives
$\delta_9\sigma_{33}=\{\bar\varepsilon_9,\eta_{24},\sigma_{25}^2\}_9$.
Therefore, by Lemma \ref{l_tb1}, we have
\begin{equation*}
  2\kappa_{10}^\ast
  =E\{\bar\varepsilon_9,\eta_{24},\sigma_{25}^2\}_{12}
  \subset E\{\bar\varepsilon_9,\eta_{24},\sigma_{25}^2\}_9
  =\delta_{10}\sigma_{34}.
\end{equation*}
This yields the assertion of Theorem \ref{IM}\ref{IM_p4110}.

\noindent
Tomohisa Inoue\\
Faculty of Economics, Shinshu University\\
Matsumoto, 390-8621, Japan\\
t\_inoue@shinshu-u.ac.jp\\

\noindent
Juno Mukai\\
Developmental Education Center, Matsumoto University\\
Matsumoto, 390-1295, Japan\\
juno.mukai@matsu.ac.jp\\

\end{document}